\documentclass[11pt,reqno]{amsart}
\usepackage{graphicx}
\pdfoutput=1
\usepackage{amsfonts,amsmath,amssymb}
\usepackage{hyperref}
\usepackage{paralist}
\usepackage[linesnumbered,ruled,vlined,norelsize]{algorithm2e}

\usepackage{thmtools}
\declaretheoremstyle[bodyfont=\normalfont]{noncursive}
\declaretheorem{theorem}
\declaretheorem[numberwithin=section]{lemma}

\declaretheorem[numberlike=lemma]{proposition}

\declaretheorem[style=noncursive,numberlike=lemma]{observation}

\newcommand{\im}{\ensuremath{\mbox{\rm Im}\,}}
\newcommand{\re}{\ensuremath{\mbox{\rm Re}\,}}

\newcommand{\CC}[1]{\mathbb{C}^{#1}}

\newcommand{\RR}[1]{\mathbb{R}^{#1}}

\newcommand{\z}{\frac{\partial}{\partial z_1}}
\newcommand{\zz}{\frac{\partial}{\partial z_2}}
\newcommand{\zzz}{\frac{\partial}{\partial z_3}}
\newcommand{\g}{\mathfrak g}

\newcommand{\lr}{\longrightarrow}

\numberwithin{equation}{section}

\newcommand{\hol}[1]{\mathfrak{hol}^{#1}\,}
\newcommand{\aut}[1]{\mathfrak{aut}^{#1}\,}

\def\Label#1{\label{#1}}

\def\1#1{\overline{#1}}
\def\2#1{\widetilde{#1}}
\def\3#1{\widehat{#1}}
\def\4#1{\mathbb{#1}}
\def\5#1{\frak{#1}}
\def\6#1{{\mathcal{#1}}}

\title[Classification of homogeneous strictly pseudoconvex hypersurfaces]{Classification of homogeneous strictly pseudoconvex hypersurfaces in $\mathbb C^3$}

\author {I. Kossovskiy}
\address{Department of Mathematics, Masaryk University in Brno//
Faculty of Mathematics, University of Vienna}
\email{kossovskiyi@math.muni.cz, ilya.kossovskiy@univie.ac.at}
\author {A. Loboda}
\address{Voronez State Technical University}
\email{lobvgasu@yandex.ru}

\begin{document}

\maketitle

\date{\today}

\begin{abstract}
Locally homogeneous strictly pseudoconvex hypersurfaces in $\CC{2}$ were classified by E.\,Cartan in 1932. In this work, we complete the classification of locally homogeneous strictly pseudoconvex hypersurfaces in $\CC{3}$.
\end{abstract}

\tableofcontents

\section{Introduction}

\subsection{Homogeneous CR-manifolds} Homogeneous real submanifolds in complex space form an important class of embedded CR-manifolds in complex space. Usually, one distinguishes between the {\em global} and the {\em local} homogeneities, respectively. Globally homogeneous (compact) CR-manifolds are, in turn, a rare find in CR-geometry (see, e.g., \cite{morimoto}), that is why most of the work on homogeneous CR-manifolds is dedicated to the locally homogeneous ones. Recall that a CR-manifolds $M\subset\CC{N}$  is called {\em locally homogeneous} if the CR-geometries at any two points in it are isomorphic to each other, that is, for every  $p,q\in M$, there exists a CR-diffeomorphism $H:\,(M,p)\lr (M,q)$ between the germs of $M$ at $p$ and $q$ respectively.
 A useful modern exposition of different notions of homogeneity in CR-geometry is given in the paper \cite{zaitsevhomog} of Zaitsev. According to \cite{zaitsevhomog}, every locally homogeneous CR-submanifold can be already assumed to be {\em real-analytic}, and the local homogeneity near one point in $M$ propagates analytically along $M$. Furthermore, the local homogeneity is equivalent to each of the following three conditions:

\medskip

(1) For every $p,q\in M$, there exists a local biholomorphism $H:\,(M,p)\lr (M,q)$ (i.e., the germs of $M$ at every two points on it are biholomorphic).

\medskip

(2) Near every point $p\in M$, there exists a locally transitive (real) Lie group action  on $M$ by biholomorphic transformation.

\medskip

(3) For each $p\in M$, the {\em infinitesimal automorphism algebra} $\mathfrak{hol}\,(M,p)$ of $M$ at $p$ (i.e., the Lie algebra of holomorphic vector fields
$$X=\left\{f_1(z)\z+\cdots+ f_N(z)\frac{\partial}{\partial z_n}\right\}$$
such that their coefficients $f_j$ are holomorphic near $p$ and $\re X$ is tangent to $M$ pointwise) is transitive on $M$ near $p$, that is, the values at $p$ of the vector fields $X\in\mathfrak{hol}\,(M,p)$ span the entire tangent space $T_p M$.

\medskip

We note here that  $\mathfrak{hol}\,(M,p)$ is precisely the algebra of all vector fields in $\CC{N}$ the flow of which consists of holomorphic transformations and preserves $M$, locally near $p$.

Since the pioneering 1932  work of E.\, Cartan \cite{cartan} who classified all locally homogeneous strictly pseudoconvex hypersurfaces in $\CC{2}$ (which  implies the classification of {\em all} locally homogeneous $3$-dimensional CR-manifolds), a lot of work has been dedicated to the general project of holomorphic classification of locally homogeneous CR-manifolds. For some complete classifications, we shall  mention the work \cite{4homog} of Beloshapka-Kossovskiy who classified locally homogeneous $4$-dimensional CR-manifolds,   and the work of Fels-Kaup ({\em Acta Math. 2008}) who classified  all Levi-degenerate locally homogeneous $5$-dimensional CR-manifolds. However, somewhat surprisingly, the classification of locally homogeneous {\em strictly pseudoconvex} hypersurfaces in $\CC{3}$ (which seems to be the most natural development of Cartan's 1932 work) has been open till present in its full generality. It is the main goal of this paper is to provide finally such a classification. We do so by treating the remaining open case when a locally homogeneous hypersurface under consideration is {\em simply homogeneous}, that is, it admits a {\em free} local Lie group action (in other words, its  {\em isotropy algebra}
$$\mathfrak{aut}\,(M,p):=\{X\in\mathfrak{hol}\,(M,p):\,\,X|_p=0\}$$
at the reference point $p$ is trivial).

Before stating our classification theorem, we outline below the progress in the classification of locally homogeneous strictly pseudoconvex hypersurfaces in $\CC{3}$ in the  case when the isotropy algebra $\mathfrak{aut}\,(M,p)$ has a positive dimension. In  the latter case, a powerfull tool for the classification is the normal form theory for Levi-nondegenerate hypersurfaces due to (Chern-)\,Moser \cite{chern}. For a locally homogeneous hypersurface, its complete normal form is simply {\em constant} along a hypersurface. Furthermore, the absense of a non-trivial stability group of a hypersurface at a point $p$ puts a lot of restrictions on the normal form at $p$. These two aspects put together make it possible to detect a few coefficients of the normal form competely determining a homogeneous hypersurface, and classify subsequently the hypersurfaces under consideration. This approach was realized  mainly by the school of A.\,Vitushkin: see e.g. Ezhov-Loboda-Schmalz \cite{els} and Loboda \cite{lobdim,lob2dim,lob1dim}. In particular, it was shown that the possible dimensions of the isotropy algebra at a point are either $10$ (the spherical case), or otherwise $2,1$, or $0$. Strictly pseudoconvex locally homogeneous hypersurfaces with stability algebras of dimension $2$ and $1$ were classified in the work \cite{lob2dim,lob1dim}, respectively.   An alternative approach in the case $\mbox{dim}\,\mathfrak{aut}\,(M,p)>0$, employing already the Cartan moving frame method and representation theory for Lie algebras, was suggested by Doubrov-Medvedev-The in \cite{dmt}. In the latter work, the authors were able (among other significant results) to revisit Loboda's classification and supplement it by one missing hypersurface in the case of $1$-dimensional isotropy. The approach in \cite{dmt} shares certain traits with the approach of Fels-Kaup in \cite{kaup} used in the Levi-degenerate case.

We shall emphasize, however, that both mentioned approaches (the one based on normal forms and the one employing the moving frame method) strictly rely on the existence of a non-trivial isotropy algebra, and are not able to provide any information on the classification when $\mathfrak{aut}\,(M,p)=0$. In this way, the simply homogeneous case treated in the present paper remained open, as discussed above. Our treatment of this case is rather close to the original Cartan's approach in \cite{cartan}. That is, we use the existence of a (locally) transitively acting $5$-dimensional Lie algebra on a simply homogeneous hypersurface $M$, and then use subsequently the strict pseudoconvexity for providing certain normal forms already for the algebras of holomorphic vector fields  acting on $M$. When doing so, we  rely on the classification of $5$-dimensional real Lie algebras due to Mubarakazyanov \cite{mubarak}.

\subsection{Main results} We now provide our results in detail. Let us recall first the construction of a natural class of locally homogeneous strictly pseudoconvex hypersurfaces in $\CC{3}$. this is the class of {\em tubes over affinely homogeneous surfaces}. Let us take an affinely homogeneous {strictly geometrically convex} (resp. concave)  surface $B\subset\RR{3}$ (the {\em base} of the tube), and then consider the tubular CR-hypersurface
$$M=B+i\RR{3}=\{z\in\CC{3}:\,\,\re z\in B\}.$$
The hypersurface $M\subset\CC{3}$ is clearly strictly pseudoconvex. If now $\mathfrak a$ is the Lie algebra of (real) affine vector fields of the kind
$$L_j(x)\frac{\partial}{\partial x},\quad x\in\RR{3},\,\,j=1,2,...,k$$  acting transitively on  $B$, and $\mathfrak b$ is the $3$-dimensional abelian algebra spanned by the vector fields $i\frac{\partial}{\partial z_j}$ (and hence generating the real shifts $z\mapsto z+ib,\,b\in\RR{3}$), then the Lie algebra $\g$ spanned by $\mathfrak b$ and the vector fields
$$L_j(z)\frac{\partial}{\partial z},\quad z\in\CC{3},\,\,j=1,2,...,k$$
clearly acts transitively already on $M$. In this way, $M$ is locally homogeneous.  A substantial part of the final classification list in \autoref{theor2} below (but not the entire list!) is obtained precisely in this way.

Our main result below shows that, somewhat surprisingly, {\em all} locally homogeneous strictly pseudoconvex hypersurfaces in $\CC{3}$ can be reduced to the above tubular CR-hypersurfaces, under the simple homogeneity assumption.

\begin{theorem}\Label{theor1}
Let $M\subset\CC{3}$ be a simply homogeneous strictly pseudoconvex hypersurface. Then $M$ is locally biholomorphic near any point $p$ in it to the tube over an affinely homogeneous srtictly geometrically convex surface $B\subset\RR{3}$.
\end{theorem}

Putting together \autoref{theor1} with earlier  classifications in the case of positive-dimensional stabilizer and existing classifications of affinely homogeneous surfaces (see  Section 4 below), we finally obtain the complete classification of {\em all} locally homogeneous strictly pseudoconvex hypersurfaces in $\CC{3}$.

\begin{theorem}\Label{theor2}
Let $M\subset\CC{3}$ be a locally homogeneous  strictly pseudoconvex hypersurface. Then $M$ is locally biholomorphic near every point $p$ in it to one the following hypersurfaces:

\medskip

\noindent{\bf 1)} $\re z_3=|z_1|^2+|z_2|^2$  \quad (the hyperquadric)

\medskip

\noindent{\bf 2)} $\re z_3=\ln(1+|z_1|^2)+b\ln(1+|z_2|^2),\quad 0<b\leq 1$

\medskip

\noindent{\bf 3)} $\re z_3=\ln(1+|z_1|^2)-b\ln(1-|z_2|^2),\quad b>0,\,b\neq 1$

\medskip

\noindent{\bf 4)} $\re z_3=\ln(1-|z_1|^2)+b\ln(1-|z_2|^2),\quad 0<b\leq 1$

\medskip

\noindent{\bf 5)} $\re z_3=\varepsilon\ln(1+\varepsilon |z_1|^2)+|z_2|^2,\quad \varepsilon=\pm 1$

\medskip

\noindent{\bf 6)} $\re z_3=\pm(\re z_1)^\alpha+(\re z_2)^2,\quad
\pm\alpha(\alpha-1)>0,\,\,\alpha\neq 2$



\medskip

\noindent{\bf 7)} $\re z_3=\re z_1\cdot\ln(\re z_1)+(\re z_2)^2$

\medskip

\noindent{\bf 8)} $\re z_1\cdot \re z_3=-\re z_1\cdot \ln(\re z_1)+(\re z_2)^2$

\medskip

\noindent{\bf 9)} $\pm(\re z_1)^2\pm(\re z_2)^2+(\re z_3)^2=1$

\medskip

\noindent{\bf 10)} $1\pm(|z_1|^2+|z_2|^2)+|z_3|^2=a|z_1^2+z_2^2+z_3^2|, \quad a>1$

\medskip

\noindent{\bf 11)} $1\pm(|z_1|^2+|z_2|^2)-|z_3|^2=a|z_1^2+z_2^2-z_3^2|, \quad 0<a<1$

\medskip

\noindent{\bf 12)} $\re z_3=(\re z_1)^\alpha(\re z_2)^\beta,\quad
\alpha\beta(1-\alpha-\beta)>0,\,\,|\alpha|,|\beta|\leq 1,\,\,|\alpha|\leq |\beta|$

\medskip

\noindent{\bf 13)} $\re z_3=\Bigl((\re z_1)^2+(\re z_2)^2\Bigr)^\beta\cdot\exp{\Bigl(\alpha\arctan\frac{\re z_2}{\re z_1}\Bigr)},\quad \beta>\frac{1}{2},\,\,(\alpha,\beta)\neq (0,1)$



\medskip

\noindent{\bf 14)} $\re z_1\cdot \re z_3=(\re z_1)^2\cdot\ln(\re z_1)+(\re z_2)^2$

\medskip

\noindent{\bf 15)} $\re z_1\cdot \re z_3=\pm(\re z_1)^\alpha+(\re z_2)^2,\quad \pm(\alpha-1)(\alpha-2)>0$

\medskip

\noindent{\bf 16)} $\Bigl(\re z_3-\re z_1\cdot\re z_2+\frac{1}{3}(\re z_1)^3\Bigr)^2=\alpha \Bigl(\re z_1-\frac{1}{2}(\re z_1)^2\Bigr)^3,\quad \alpha< -\frac{8}{9}$

\medskip

\noindent{\bf 17)} $\re z_3=\re z_1\cdot(\alpha\ln(\re z_1)-\ln(\re z_2)), \quad \alpha>1$

\medskip

Here each of the hypersurfaces {\bf 1)}-- {\bf 17)} shall be considered near an arbitrary strictly pseudoconvex point $o$ in it.  Furthermore, the dimensions of the stability algebras for the hypersurfaces {\bf 1)}-- {\bf 17)} are as follows:

\smallskip

$\bullet$ $\mbox{dim}\,\mathfrak{hol}\,(M,o)=10$ for the hyperquadric {\bf 1)},

\smallskip

$\bullet$    $\mbox{dim}\,\mathfrak{hol}\,(M,o)=2$ for hypersurfaces {\bf 2)}-- {\bf 5)},

\smallskip

$\bullet$ $\mbox{dim}\,\mathfrak{hol}\,(M,o)=1$ for hypersurfaces {\bf 6)}-- {\bf 11)}, and

\smallskip

$\bullet$ $\mbox{dim}\,\mathfrak{hol}\,(M,o)=0$ for hypersurfaces {\bf 12)}-- {\bf 17)}.

\smallskip

Finally, any two hypersurfaces in the list {\bf 1)}-- {\bf 17)} are pairwise locally holomorphically inequivalent.
\end{theorem}

\medskip

\begin{center} 
\bf Acknowledments
\end{center}

\medskip

The first author was supported by the GACR (Czech Grant Agency, grant 17-19427S) and the FWF (Austrian Science Fund) during the preparation of this paper. The second author was support by the RFFI (Russian Foundation for Basic Research) grant 17�01�00592-a.

\section{Principal approach}

As discussed above, we are concerned with the case of a $5$-dimensional Lie algebra transitively acting on a strictly pseudoconvex real hypersurface by CR-transformations. Our approach to the classification then is based on  realizing (abstract) $5$-dimensional Lie algebras acting transitively on a real hypersurface by holomorphic vector fields, and finding subsequently appropriate normal forms for such realizations. In accordance with that, we will make  extensive use of the classification of (abstract) $5$-dimensional Lie algebras up to an isomorphism. The latter was obtained by Mubarakazynov in \cite{mubarak}. For convenience of the reader, we give Mubarkazynov's list of $5$-dimensional Lie algebras in Appendix A.

In what follows, we fix the following notations and conventions: $(z_1,z_2,z_3)$ denote the coordinates in $\CC{3}$, $M$ denotes a (real-analytic) strictly pseudoconvex simply homogeneous  near a point $p\in M$ hypersurface in $\CC{3}$, and $\g$ a $5$-dimensional Lie algebra of holomorphic vector fields acting on $M$ locally transitively near the point $p$. We have, accordingly,
\begin{equation}\Label{holg}
\hol{}(M,p)=\g,\quad \aut{}(M,p)=0.
\end{equation}
We also denote by $X_j,\,j=1,...,5$ a collection of holomorphic vector fields from $\g$ defined in some neighborhood $U$ of the point $p$ and pointwice linearly  independent in $U$ (over $\RR{}$). Thus, we have $\g=\mbox{span}\,\{X_1,...,X_5\}$ pointwice in $U$. We call such a collection {\em a basis} for $\g$.

We also make use of the following

\smallskip

\noindent{\bf Convention.} Solely for the purposes of the proof of \autoref{theor1}, we assume tubular CR-hypersurfaces to be invariant under the {\em real} shifts
$$z\mapsto z+a,\quad a\in\RR{3}$$
(unlike the set up in the Introduction). Accordingly, a tubular real hypersurface looks as $$M=\RR{3}+iB,$$ where $B$ is a surface in $\RR{3}$, and $\mathfrak{hol}\,(M,0)$ contains the abelian subalgebra spanned by $\frac{\partial}{\partial z_j},\,j=1,2,3$, which fits better our normalization procedure for the Lie algebras of holomorphic vector fields.

Our goal is, based on properties of $\g$ as an abstract Lie algebra, bring the basis vector fields $X_1,...X_5$ (and hence $\g$ itself) by a series of biholomorphic transformations to a certain normal form, in which $\g$ is "maximally simplified". The latter makes it possible to either obtain a contradiction with the strong pseudoconvexity of $M$, or to  recognize $M$ (up to a local biholomorphic equivalence) as a tube over an affinely homogeneous hypersurface in $\RR{3}$.

Let us make the following useful

\begin{observation}\Label{generic}
In order to prove the assertion of \autoref{theor2} at a reference point $p\in M$, it is obviously sufficient to prove the same assertion at any other point $s\in M$ close by $p$ (in view of the local homogeneity of $M$). In view of that, we may change during the proof the reference point under consideration.
\end{observation}

We make use of the following two important propositions.

\begin{proposition}\Label{key2}
Let $X,Y\in\g$ be two vector fields such that

\smallskip

\noindent (i) $X,Y$ linearly independent over $\RR{}$ at a point $q\in M$;

\smallskip

\noindent (ii) the real span of $X,Y$ is a subalgebra in $\g$ (that is, $[X,Y]\in\mbox{span}_{\,\RR{}}\{X,Y\}$ at every point).

\smallskip

\noindent
Then $X,Y$ are also linearly independent over $\CC{}$ at $q$.
\end{proposition}

\begin{proof}
Assume, by contradiction, that $X,Y$ span a $1$-dimensional complex plane at $q$. Consider then the orbit at $q$ of the action of the above $2$-dimensional subalgebra spanned by $X,Y$ (denote the latter by $S$). We have $S\subset M$ and $T_q S=\mbox{span}_{\RR{}}\{X_q,Y_q\}$, so that by assumption the plane $T_qS$ is a $1$-dimensional complex plane and thus $S$ is a complex curve (since $S$ is homogeneous). Since $S\subset M$, this gives a contradiction with the strict pseudoconvexity of $M$, and proves the proposition.
\end{proof}

\begin{proposition}\Label{key3}
Let $X,Y,Z\in\g$ be three commuting vector fields which are linearly independent over $\RR{}$ at a point $q\in M$. Then
 $X,Y,Z$ are also linearly independent over $\CC{}$ at $q$.
\end{proposition}

\begin{proof}
Assume, by contradiction, that $\mbox{dim\,span}_{\,\CC{}}\,\{X_q,Y_q,Z_q\}=2$ (complex dimension $1$ is excluded since the real span has dimension $3$). Then the orbit of the $3$-dimensional abelian algebra $\mathfrak a$ spanned by $X,Y,Z$ is $3$-dimensional real manifold $N$ contained in a complex hypersurface $S$ (which is the orbit of the complex action of $X,Y,Z$ near $q$). This implies that $N$ has a $1$-dimensional complex tangent at $q$, which means that there exists $\RR{}$-linearly independent vector fields $U,V\in\mathfrak a$ with $V_q=iU_q$. Since $U,V$ form a $2$-dimensional abelian algebra,  the latter contradicts \autoref{key2}. This proves the proposition.









\end{proof}



\section{Proof of Theorem 1}

In this section, we apply \autoref{key2} and \autoref{key3} (read together with \autoref{generic}) to prove \autoref{theor1}. More precisely, we show that any Lie algebra from Mubarakazyanov's list (see Appendix A)  realized as a $5$-dimensional Lie algebra can act locally transitively by holomorphic transformations on a strictly pseudoconvex hypersurface $M\subset\CC{3}$ only if the latter is (up to a local biholomorphic equivalence) a tube over an affinely homogeneous strictly geometrically convex (resp. concave) hypersurface $S\subset\RR{3}$.

We use the set-up and notations of Section 2. 
Let us first observe the following important fact.

\begin{proposition}\Label{ideal}
In the notations and setting of Section 2, if the algebra $\g$ contains a $3$-dimensional abelian ideal $\mathfrak a$, then $M$ is biholomorphically equivalent (locally near $p$) to the tube over an affinely homogeneous strictly geometrically convex hypersurface $S\subset\RR{3}$.
\end{proposition}

\begin{proof}
Let us choose a basis for $\g$ in such a way that $\mathfrak a$ is spanned by $X_1,X_2,X_3$. According to \autoref{key3} (and \autoref{generic}), the vector fields $X_1,X_2,X_3$ can be assumed to have compelx rank $3$ at $p$. Hence there exists a biholomorphic coordinate change near $p$ mapping $p$ into the origin and $X_1,X_2,X_3$ onto $\z,\zz,\zzz$ respectively. Since $\mathfrak a$ is an ideal in $\g$, we have
$$\left[\frac{\partial}{\partial z_l},X\right]=\sum_{j=1}^3 \alpha_{lj}\frac{\partial}{\partial z_j}$$
for any $X\in\g$, where $\alpha_{lj}$ are real coefficients. Thus all the derivatives in $z_l$ of the components of $Y$ are real constants, and we conclude that  vector fields $X_4,X_5$ completing $X_1,X_2,X_3$ to a basis have the form:
\begin{equation}\Label{X45}
X_l=\sum_{j=1}^3 a_{lj}\frac{\partial}{\partial z_j}+Z\cdot B_l\cdot\frac{\partial}{\partial Z},\quad l=4,5.
\end{equation}
Here $a_{lj}$ are complex constants, $Z=(z_1,z_2,z_3)$, $\partial Z=\left(\z,\zz,\zzz\right)^T$, and $B_l$ are constant {\em real} $3\times 3$ matrices (in particular, $X_4,X_5$ are affine vector fields). By adding to $X_4,X_5$ appropriate real linear combinations of $X_1,X_2,X_3$, we further achieve $a_{lj}\in i\RR{}$. Moreover, the fact that $\mbox{rank}_{\,\RR{}}\,\{X_1,...,X_5\}=5$ at $0$ implies that the real rank of the matrix $\{a_{lj}\}$ equals $2$. All the latter precisely means that the orbit of $\g$ at $0$ (which coincides with $M$) is the tube over an affinely homogeneous surface $S$, which is in turn the orbit at $0$ of
$$\sum_{j=1}^3 \im a_{lj}\frac{\partial}{\partial y_j}+Y\cdot B_l\cdot{\partial Y},\quad l=4,5, \quad Y=\im Z.$$
This proves  the proposition.

\end{proof}

We now refer directly to Mubarakazyanov's list describe in Appendix A.

\smallskip

\noindent {\bf Case 1: decomposable solvable algebras.}
It is possible to see from Mubarakazyanov's classification (see the first table in Appendix A) that all the decomposable solvable Lie algebras, with the exeption of $\mathfrak m_{26}$, contain a $3$-dimensional abelian ideal.
 Applying now \autoref{ideal}, we  conclude that any possible stongly pseudoconvex orbit in Case 1 is  biholomorphic to the tube over an affinely homogeneous strictly geometrically convex hypersurface $S\subset\RR{3}$, with possibly the only exception of $$\g=\mathfrak m_{26}.$$
 The latter exceptional algebra shall be be treated separately. We claim that for this algebra there are no simply homogeneous strictly pseudoconvex orbits. Indeed, the nontrivial commuting relations for $\g$ are:
 $$[X_2,X_3]=X_1,\,\,[X_1,X_4]=2qX_1,\,\,[X_2,X_4]=qX_2-X_3,\,\,[X_3,X_4]=X_2+qX_3,\,\,q\geq 0.$$
 We introduce the vector fields
 $$X_2':=X_2+iX_3,\quad X_3':=X_2-iX_3.$$
 Then the nontrivial commuting relations involving $X_2,X_3$ turn into
 $$[X_2',X_3']=-2iX_1,\quad [X_2',X_4]=(q+i)X_2', \quad [X_3',X_4]=(q-i)X_3'.$$
 Note that both triples $X_1,X_2,X_5$ and $X_1,X_3,X_5$ form abelian subalgebras, thus both triples have complex rank $3$ at the reference point $p$ by \autoref{key3}. This implies that at least one of the triples $X_1,X_2',X_5$ and $X_1,X_3',X_5$ (say the first one) has complex rank $3$ at $p$. We then straighten the commuting vector fields near $p$ and get:
 $$X_1=\z,\quad X_2'=\zz, \quad X_5=\zzz.$$
 Now from the commuting relations of the remaining fields with $X_1,X_2',X_5$ we easily get:
 $$X_3'=(-2iz_2+c)\z+a\zz+b\zzz,\,\,X_4=2qz_1\z+(q+i)z_2\zz+d\zzz$$
 ($a,b,c,d$ - constants). Further, the commuting relation for $X_3',X_4$ gives (by considering the components $\z,\zz,\zz$ respectively):
 $$c=0,\quad a=0, \quad b=0.$$
 We finally get:
 $$X_2=-iz_2\z+\frac{1}{2}\zz,\quad X_3=z_2\z+\frac{1}{2i}\zz.$$
 We claim now that all orbit of the $5$-dimensional algebra $\g$ obtained above have an additional holomorphic symmetry and thus are not simply homogeneous. Indeed, consider the vector field
 $$Y:=iz_2\zz.$$
 It is easy to check that
 $$[Y,X_1]=[Y,X_4]=[Y,X_5]=0,\quad [Y,X_2]=X_3,\quad [Y,X_3]=-X_2.$$
 The latter means $[Y,\g]\subset\g$. At the same time, at points on the hypersurface $\Sigma:=\{z_2=0\}$ the algebra $\g$ has the full rank $5$, while the vector field $Y$ vanishes. This means that the orbits through these points of the algebras $\g$ and $\g\oplus\CC{}Y$ coincide. Since a generic orbit must intersect $\Sigma$, this proves that  all the orbits of $\g$ are invariant under the action of $Y\not\in\g,$ as required.

\smallskip

\noindent {\bf Case 2: Nondecomposable solvable algebras.} These algebras form the largest subset in Mubarakazyanov's list, and we step-by-step go through the list of algebras.

First, all the algebras in the ranges $g_1$ -- $g_{18}$, $g_{30}$ -- $g_{35}$, and $g_{38}$ -- $g_{39}$ contain the $3$-dimensional abelian ideal spanned by $e_1,e_2,e_3$. Next, all the algebras in the ranges $g_{19}$ - $-g_{24}$ and $g_{27}$ -- $g_{29}$ contain the $3$-dimensional abelian ideal spanned by $e_1,e_3,e_4$.
According to \autoref{ideal}, the orbits of all the mentioned algebras then appear to be locally biholomorphic to the tube over an affinely homogeneous hypersurface. We are left with the four exceptional algebras: $g_{25}, g_{26}, g_{36}, g_{37}$.

\smallskip

\noindent{\bf Subcase  $\g=g_{36}$}. In the latter case, $\g$ is realized by holomorphic in a neighborhood of $p$ vector fields $X_i,\,i=1,..,5$ with the following only nontrivial commuting relations:

$$[X_2,X_3]=X_1,\,\,[X_1,X_4]=X_1,\,\,[X_2,X_4]=X_2,\,\,[X_2,X_5]=-X_2,\,\,[X_3,X_5]=X_3.$$

According to \autoref{key2}, the complex rank at $p$ of $X_1,X_2$ equals $2$, thus these two vector fields can be simultaneously straightened near $p$:
$$X_1=\z,\,\,\,X_2=\zz.$$
Taking into account the commuting relations of $X_1,X_2$ with $X_3$, we conclude that in such coordinates $X_3$ has the form:
\begin{equation}\Label{h0}
X_3=(z_2+f(z_3))\z+g(z_3)\zz+h(z_3)\zzz.
\end{equation}
First, consider the situation  $h\equiv 0$. Then we do the variable change $z_2^*=z_2+f(z_3)$ and $X_3$ becomes
$$X_3= z_2\z+A(z_3)\zz.$$
We then work out $X_4$. The commuting relations with $X_1,X_2$ give:
\begin{equation}\Label{h00}
X_4=(z_1+f(z_3))\z+(z_2+g(z_3))\zz+h(z_3)\zzz.
\end{equation}
Note that $h\equiv$ is not possible for the latter identity, since otherwice $X_1,X_2,X_3,X_4$ span a subalgebra of vector fields non of which has the $\zzz$ component, thus their orbit at $p$ lies in $z_3=const$, hence it coincides with $z_3=const$ and we obtain a contradiction with the strict pseudoconvexity. In view of that, after possibly changing the base point $p$, we pay assume $h(z_3)\neq 0$ at $p$ in \eqref{h00} and straighten the vector field $h(z_3)\zzz$. This means
$$X_4=(z_1+f(z_3))\z+(z_2+g(z_3))\zz+\zzz.$$
A variable change $z_1^*=z_1+\psi(z_3)$ for appropriate $\psi$ allows to further make $f=0$ (this is accomplished by choosing $\psi$ such that $\psi'-\psi+f=0$). Using now $[X_3,X_4]=0$, we get first $g=0$ by considering the $\z$ component, and then $A-A'=0$ by considering the $\zzz$ component. This finally gives
$$X_3= z_2\z+\alpha e^{z_3}\zz, \quad X_4=z_1\z+z_2\zz+\zzz$$
($\alpha$ here is a constant).
Finally, consider the vector field
$$X_5':=2X_4+X_5.$$
We have $$[X_1,X_5']=2X_1, \,\,[X_2,X_5']=X_2,\,\,[X_3,X_5']=X_3.$$
The commuting relations with $X_1,X_2$ give
$$X_5'=(2z_1+f(z_3))\z+(z_2+g(z_3))\zz+h(z_3)\zzz.$$
Then the commuting relation with $X_3$ gives, by considering the $\zz$ component: $\alpha e^{z_3}h=0$, so either:

(i) $h=0$, so the vector fields $X_1,X_2,X_3,X_5'$ span a subalgebra of vector fields non of which has the $\zzz$ component, and repeating the argument above, we obtain a contradiction with the strict pseudoconvexity,

or

(ii) $\alpha=0$, and the vector fields $X_1,X_3$ provide a contradiction with \autoref{key2}.

This means that $h\not\equiv 0$ in \eqref{h0}. Shifting if necessary the base point $p$, we may assume $h(z_3)\neq 0$ at $p$ in \eqref{h0} and straighten the vector field $h(z_3)\zzz$.  Thus $X_3$ becomes:
$$X_3=(z_2+f(z_3))\z+g(z_3)\zz+\zzz.$$
Acting as above, we further find functions $\phi(z_3),\psi(z_3)$ such that the variable change
$$z_1^*=z_1+\phi(z_3),\quad z_2^*=z_2+\psi(z_3),\quad z_3^*=z_3$$ annihilates $f,g$, i.e. $X_3$ becomes
$$X_3=z_2\z+\zzz.$$
Next, for $X_4$ we use commuting relations with $X_1,X_2$ and get:
$$
X_4=(z_1+f(z_3))\z+(z_2+g(z_3))\zz+h(z_3)\zzz.$$
The commuting relation with $X_3$ implies (by considering the $\zzz$ component): $h'=0$, so $h=c$ (a constant). Similarly, for $X_5$ we get:
$$X_5=F(z_3)\z+(-z_2+G(z_3))\zz+H(z_3)\zzz.$$
Now the commuting relation with $X_3$ implies (by considering the $\zzz$ component): $H'=1$, so $H=z_3+C$ (a constant).
Finally, the commuting relation of $X_4$ and $X_5$ implies (by considering the $\zzz$ component): $c=0$.

We now consider the subalgebra spanned by $X_1,X_2,X_4$. Recall that
\begin{equation}\Label{ev1}
X_1=\z,\quad X_2=\zz,\quad X_4=(z_1+f(z_3))\z+(z_2+g(z_3))\zz.
\end{equation}
Integrating the action of \eqref{ev1} near $p$ gives the {\em flat} orbit
$$z_3=const, \quad \im (z_1- z_2)=const,$$ which contains, in particular, complex lines. This gives a contradiction with the strict pseudoconvexity.

We finally conclude that there are no strictly pseudoconvex orbits in the case $\g=g_{36}.$

\smallskip

\noindent{\bf Subcase  $\g=g_{25}$}.
Here the nontrivial commuting relations are
$$[X_2,X_3]=X_1,\,\,[X_1,X_5]=2qX_1,\,\,[X_2,X_5]=qX_2-X_3,\,\,[X_3,X_5]=X_2+qX_3,\,\,[X_4,X_5]=pX_4,\,\,p\neq 0.$$
Arguing very similarly to the case $\g=g_{4.9}\oplus g_1$, we conclude that in appropriate local holomorphic coordinates $\g$ can be represented as:
\begin{equation}\Label{nf25}
\begin{aligned}
&X_1=\z,\,\,X_2=-iz_2\z+\frac{1}{2}\zz,\,\, X_3=z_2\z+\frac{1}{2i}\zz,\,\,X_4=\zzz,\\
&X_5=2qz_1\z+(q+i)z_2\zz+qz_3\zzz.
\end{aligned}
\end{equation}
Introducing as above the vector field $$Y:=iz_2\zz,$$
 we check that
 $$[Y,X_1]=[Y,X_4]=[Y,X_5]=0,\quad [Y,X_2]=X_3,\quad [Y,X_3]=-X_2.$$
 The latter means $[Y,\g]\subset\g$, and arguing as in Case 1 we conclude that all the orbits of $\g$ are invariant under the action of $Y\not\in\g,$ so that the orbits of $\g$ are not simply homogeneous in the case $\g=g_{25}$.

\smallskip

\noindent{\bf Subcase  $\g=g_{26}$}.
Here the nontrivial commuting relations are
$$[X_2,X_3]=X_1,\,\,[X_1,X_5]=2qX_1,\,\,[X_2,X_5]=qX_2-X_3,\,\,[X_3,X_5]=X_2+qX_3,\,\,[X_4,X_5]=\epsilon X_1+2qX_4,$$
where $q\in\RR{},\,\,\epsilon=\pm 1.$
Arguing, again, very similarly to the case $\g=g_{4.9}\oplus g_1$, we conclude that in appropriate local holomorphic coordinates $\g$ can be represented as:
\begin{equation}\Label{nf26}
\begin{aligned}
&X_1=\z,\,\,X_2=-iz_2\z+\frac{1}{2}\zz,\,\, X_3=z_2\z+\frac{1}{2i}\zz,\,\,X_4=\zzz,\\
&X_5=(2qz_1+\epsilon z_3)\z+(q+i)z_2\zz+2qz_3\zzz.
\end{aligned}
\end{equation}
Introducing as above the vector field $$Y:=iz_2\zz,$$
 we check that
 $$[Y,X_1]=[Y,X_4]=[Y,X_5]=0,\quad [Y,X_2]=X_3,\quad [Y,X_3]=-X_2.$$
 The latter means $[Y,\g]\subset\g$, and  arguing as in Case 1 we conclude that all the orbits of $\g$ are invariant under the action of $Y\not\in\g,$ so that the orbits of $\g$ are not simply homogeneous in the case $\g=g_{26}$ as well.

\smallskip

\noindent{\bf Subcase  $\g=g_{37}$}. Here we have the following  nontrivial commutation relations:
$$[X_2,X_3]=X_1,\,\,[X_1,X_4]=2X_1,\,\,[X_2,X_4]=X_2,\,\,[X_3,X_4]=X_3,\,\,[X_2,X_5]=-X_3,\,\,[X_3,X_5]=X_2.$$
Using \eqref{key2}, we straighten $X_1,X_2$ so that:
$$X_1=\z,\,\,\,X_2=\zz.$$
Taking into account the commuting relations of $X_1,X_2$ with $X_3$, we conclude that in such coordinates $X_3$ has the form:
\begin{equation}\Label{xxx}
X_3=(z_2+f(z_3))\z+g(z_3)\zz+h(z_3)\zzz.
\end{equation}
Under the assumption $h\not\equiv 0$ in \eqref{xxx}, it is not difficult by arguing as above to further simplify $X_3$ to become:
$$
X_3=z_2\z+\zzz.
$$
After that, by using the commuting relations for $X_4$, $X_5$ in a straightforward manner as shown above, we compute that $X_4,X_5$ have the form:
$$X_4=(2z_1+Az_3)\z+  (z_2+A)\zz+z_3\zzz,$$
$$X_5=\left(\frac{1}{2}(z_3^2-z_2^2)+\frac{1}{2}A^2\right)\z+  z_3\zz-(z_2+A)\zzz$$
(after a shift in $z_1$). Here $A=a+bi$ is a complex constant. Further, it is convenient to shift $z_2$ by $A$ which finally gives:
$$X_3=(z_2-A)\z+\zzz,$$
$$X_4=(2z_1+Az_3)\z+  z_2\zz+z_3\zzz,$$
$$X_5=\left(\frac{1}{2}(z_3^2-z_2^2)+Az_2\right)\z+  z_3\zz-z_2\zzz,$$
and $X_1,X_2$ are as above.
It is straightforward to check then that the real parts of all these five vector fields are tangent to the 1-parameter family of real hyperquadrics
$$y_1=x_3(y_2-b)-ay_2+N(y_2^2+y_3^2).$$
The latter means that all the strictly pseudoconvex orbits of the algebra $\g$ are automatically {\em spherical} under the assumption $h\not\equiv 0$ in \eqref{xxx}.

If, otherwise, $h\equiv 0$ in \eqref{xxx}, we can claim that for the vector field
$$X_4=u\z+v\zz+w\zzz$$
we have $w(p)\neq 0$ (since otherwise the orbit at $p$ of the algebra spanned by $X_1,X_2,X_3,X_4$ has a $2$-dimensional complex tangent at $p$ and hence is a complex hypersurface itself). This allows us, by arguing as above, to simplify $X_4$ to be
$$X_4=2z_1\z+z_2\zz+\zzz$$
(while $X_1,X_2$ stay the same)
It is convenient for us now to do the substitution $z_3^*=e^{z_3}$ which turns $X_4$ into
$$X_4=2z_1\z+z_2\zz+z_3\zzz.$$
Having $X_1,X_2,X_3$ normalized as above, it is straightforward to compute by using the commuting relations for $X_3,X_5$ that the latter two vector fields look as
$$X_3=z_2\z+A\zz,$$
$$X_5=(Bz_3^2-\frac{1}{2}z_2^2)\z+Cz_2\zz+Dz_3\zzz.$$
Finally, by quadratic changes of variables we are able to simplify $X_5$ to be
$$X_5=Cz_2\zz+Dz_3\zzz$$
keeping the other vector fields unchanged.

We now use a similar idea to that for the algebras $g_{25},g_{26}$ and notice that the vector field
$$Y:=z_2\zz$$
satisfies $$[\g^{\CC{}},Y]\subset \g^{\CC{}}$$
(where $\g^{\CC{}}$ is the complexified algebra of holomorphic vector fields). By considering orbits through points with $z_2=0$ and generic $z_1,z_3$ (where the algebra $\g$ has the full rank $5$), we see that   the external complexification $M^{\CC{}}\subset\CC{3}\times\overline{\CC{}}$ of $M$ has an infinitesimal automorphism algebra of dimension at least $6$. Since $\g$ is a real form of $\g^{\CC{}}$, we conclude finally that $M$ is not simply homogeneous.

In view of that, the algebra $g_{37}$ does not have any simply homogeneous
strictly pseudoconvex orbits.

\smallskip

\noindent {\bf Case 3: decomposable nonsolvable algebras.} The latter case occurs  when the minimal decomposition of $\g$ contains a $3$-dimensional simple term.

We start with the situation when the minimal decomposition of $\g$ contains $3$ terms, i.e. $\g$ is the sum of a $3$-dimensional simple term $\mathfrak{a}$ and a $2$-dimensional abelian term $\mathfrak b$. Fix a reference point $p$ on the orbit and consider the intersection of $\mathfrak a$ evaluated at $p$ with the complexification of $\mathfrak b$ evaluated at $p$. In view of dimension reasons, such intersection is nonempty, hence there exists $W\in\mathfrak a$ such that the complex rank of $\CC{}W\oplus\mathfrak b$ at $p$ equals $2$. On the other hand,   $\CC{}W\oplus\mathfrak b$ is abelian, and we obtain a contradiction with \autoref{key3}.

We conclude finally that the only possibility of obtaining a strictly pseudoconvex orbit is the following: $\g$ is the sum of   a $3$-dimensional simple term $\mathfrak{a}$ and the $2$-dimensional nonabelian term $\mathfrak g_2$. This leads to the following two subcases.

\smallskip

\noindent{\bf Subcase  $\g=\mathfrak m_{16}=\mathfrak{su}(1,1)\oplus\g_2$}. In this case, we have the following commutation relations:
$$[X_1,X_2]=X_1,\,\,[X_1,X_3]=2X_2,\,\,[X_2,X_3]=X_3,\,\,[X_4,X_5]=X_4.$$
Using \eqref{key2}, we straighten $X_3,X_4$ so that:
$$X_3=\z,\,\,\,X_4=\zz.$$
Taking into account the commuting relations of $X_3,X_4$ with $X_5$, we conclude that  $X_5$ has the form:
\begin{equation}\Label{XXX}
X_5=f(z_3)\z+(z_2+g(z_3))\zz+h(z_3)\zzz.
\end{equation}
We claim that $h\not\equiv 0$ in \eqref{XXX}. Indeed, arguing by contradiction, we consider first $[X_1,X_4]=0$ and conclude that components of $X_1$ do not depend on $z_2$. Next, we consider the last component of the identity $[X_1,X_5]=0$ and conclude that $f\cdot H_{z_1}=0$ (if $X_1=F\z+G\zz+H\zzz$), so that $H_{z_1}\equiv 0$ ($f\equiv=0$ is not possible, since then $X_5$ is a product of $X_4$ with a holomorphic function, which implies the holomorphic degeneracy of the orbit). Finally, considering the last component in $[X_1,X_3]=2X_2$, we conclude that the last component of $X_2$ vanishes identically. All together, this implies that the for the four vector fields $X_2,X_3,X_4,X_5$ (forming a subalgebra in $\g$) their last component is zero, thus the orbit of the subalgebra at every point is a complex plane, which is a contradiction with the Levi-degeneracy of the orbit. This prove the claim

In this way, we assume $h\not\equiv 0$ in \eqref{XXX}, and then move to a nearby point in $M$ in order to have $h(p)\neq 0$. It is not difficult then, by arguing as above, to further simplify $X_5$ to become:
$$
X_5=z_2\zz+\zzz.
$$
Further, the substitution $z_3^*=e^{z_3}$ makes
$$
X_5=z_2\zz+z_3\zzz.
$$
After that, considering the commutators of $X_2$ with $X_3,X_4$ and $X_5$, it is not difficult to get:
\begin{equation}\Label{XX}
X_2=-z_1\z+Az_3\zz+Bz_3\zzz.
\end{equation}
In case $B\neq 0$ in \eqref{XX}, by mean of a linear substitution (the non-identical part of which has the form $z_2\mapsto z_2+\alpha z_3$), it is possible to further achieve
$$X_2=-z_1\z+Bz_3\zzz$$
(preserving the form of the other vector fields).
Finally, we figure out the form of $X_1$. Commutation relations of $X_1$ with $X_4,X_3,X_5$ and $X_2$ respectively yield (after a straightforward calculation):
$$X_1=z_1^2\z+Cz_3\zz-2Bz_1z_3\zzz,$$
where $$C(B+1)=0.$$

In case $C=0$, we observe that all the vector fields $$Y_\lambda:=\lambda z_3\zzz,\quad\lambda\in\CC{}$$ commute with $\g$. For an appropriate choice of $\lambda$, the value of $Y_\lambda$ at $p$ must lie in $T_p M$ (recall that $z_3\neq 0$ at $p$). This means that $M$ is invariant under the action of $\g\oplus \CC{}Y_\lambda$, so that $M$ is not simply homogeneous.

In case $C\neq 0,\,B=-1$, we observe the following. The linear map
$$\sigma(z_1,z_2,z_3):=(z_2,z_1,z_3)$$
preserves the subalgebra spanned by $X_1,X_2,X_3,X_4$, while the vector field $X_5$ becomes
$$X_1':=Cz_3\z+ z_2^2\zz+2z_2z_3\zzz.$$
It is straightforward to check then that we have $[X_1',\g]\subset\g\oplus\CC{}X_1'$. Then we argue as above and note that, for example, at points with $z_2=0,\,z_1\neq 0$ the algebra $\g$ has full rank while the value of $X_1'$ at $p$ lies in the span of $X_1,..,X_5$, so that the orbits at such points are invariant under  the action of $\g\oplus \CC{}X_1$ and $X_1$ is thus an additional infinitesimal automorphism. The latter applies, by uniqueness, to all the orbits. We again conclude that neither of the orbits is simply homegeneous.

Finally, in the case $B= 0$ in \eqref{XX}, we may replace $X_2$ by $X_2+X_5$ and then argue identically to the case $B\neq 0$ to simplify $X_2$.  Since $[X_1,X_2+X_5]=[X_1,X_2]$, we get an identical representation to the above for $X_1$ and again conclude, that the orbits are not simply homogeneous.

\smallskip

\noindent{\bf Subcase  $\g=\mathfrak m_{17}=\mathfrak{su}(2)\oplus\g_2$}. In this case, we have the following nontrivial commutation relations:
$$[X_1,X_2]=X_3,\,\,[X_1,X_3]=-X_2,\,\,[X_2,X_3]=X_1,\,\,[X_4,X_5]=X_4.$$
Let us introduce the new vector fields
$$X_1':=X_1+iX_3,\quad X_3':=X_1-iX_3.$$
Then the respective modified (nontrivial) commutation relations are:
$$[X_1',X_3']=2iX_2,\,\,[X_1',X_2]=-iX_1',\,\,[X_2,X_3']=-iX_3',\,\,[X_4,X_5]=X_4.$$
We observe that the modified commutation relations are very similar to that for $\g=\mathfrak{su}(1,1)\oplus\g_2$. In accordance with that, arguing identically to the case $\g=\mathfrak{su}(1,1)\oplus\g_2$, we can simplify $\g$ in order that the (modified) basic vector fields look as:
$$X_1'=-z_1^2\z+Cz_3\zz-2iBz_1z_3\zzz, \quad X_2=iz_1\z+Bz_3\zzz,$$
$$X_3'=\z,\quad X_4=\zz,\quad X_5=z_2\zz+z_3\zzz,$$
where the complex constants $B,C,\,B\neq 0$ satisfy:
$$C(B-i)=0.$$

In case $C=0$, the basic vector fields finally look as
$$2X_1=(1-z_1^2)\z-2iBz_1z_3\zzz, \quad X_2=iz_1\z+Bz_3\zzz,$$
$$2X_3=i(1+z_1^2)\z-2Bz_1z_3\zzz,\quad X_4=\zz,\quad X_5=z_2\zz+z_3\zzz.$$
Arguing now again identically to the situation of $\g=\mathfrak{su}(1,1)\oplus\g_2$ and employing the vector field $$Y_\lambda:=\lambda z_3\zzz,\quad\lambda\in\CC{},$$ we similarly conclude that the orbits in this case are not simply homogeneous.

In case $C\neq 0,\,B=i$, the basic vector fields finally look as
$$2X_1=(1-z_1^2)\z+Cz_3\zz+2z_1z_3\zzz, \quad X_2=iz_1\z+Bz_3\zzz,$$
$$2X_3=i(1+z_1^2)\z-iCz_3\zz-2iz_1z_3\zzz,\quad X_4=\zz,\quad X_5=z_2\zz+z_3\zzz.$$
We claim that the Levi-nondegenerate orbits of the latter algebra are {\em not} strictly pseudoconvex. For that, we first note that at all points in $\CC{3}$ with $z_1=z_2=0,\,z_3\neq 0$, the real rank of the vector fields $X_1,...,X_5$ is $5$. This means that a generic orbit of $\g$ intersects the complex line
$$L:=\{z_1=z_2=0\}.$$
At the same time, at all points in $L$ the values of the commuting vector fields $X_2$ and $X_5$ are linearly dependent over $\CC{}$. The latter contradicts \autoref{key2} and proves that neither of the orbits is strictly pseudoconvex.

 \smallskip

\noindent {\bf Case 4: nondecomposable nonsolvable algebras.} According to Mubarakazyanov's classification, there is a unique nondecomposable nonsolvable $5$-dimensional Lie algebra, namely $\mathfrak g_5$ (see Appendix A). This particular algebra and its orbits in $\CC{3}$ were considered in the recent paper \cite{atanovloboda} of Atanov-Loboda, and the outcome of Case 4 can be read from the latter paper. However, we provide, for completeness, an alternative proof here.

  Nontrivial commutation relations in $\g=\g_5$ look as:
\begin{equation}\Label{g5}
\begin{aligned}
&[X_1,X_2]=2X_1,\,\,[X_1,X_3]=-X_2,\,\,[X_2,X_3]=2X_3,\,\,[X_1,X_4]=X_5,
\\
&[X_2,X_4]=X_4,\,\,[X_2,X_5]=-X_5,\,\,[X_3,X_5]=X_4.
\end{aligned}
\end{equation}
According to \eqref{key2}, we can straighten the commuting vector fields $X_1$ and $X_5$, so that $$X_1=\z,\,\,\,X_5=\zzz.$$
Using the commutation relations for $X_4$ with $X_1,X_5$, we get:
\begin{equation}\Label{xxxx}
X_4=f(z_2)\z+g(z_2)\zz+(z_1+h(z_2))\zzz.
\end{equation}
Similarly, for $X_2$ we get:
\begin{equation}\Label{xx}
X_2=(2z_1+u(z_2))\z+v(z_2)\zz+(z_3+w(z_2))\zzz
\end{equation}
(here $f,g,h,u,v,w$ are all holomorphic functions depending on $z_2$ only near the reference point $p=(p_1,p_2,p_3)$). Obviously, the identity $g(p_2)=v(p_2)=0$ is not possible, since then the real span of the subalgebra spanned by $X_1,X_2,X_4,X_5$ at $p$ is  the complex $2$-plane $z_2=const$, so that the orbit of this subalgebra is a complex surface, which is a contradiction. We now come to a case distinction.

Assume first that $g(p_2)\neq 0$, and then perform a transformation in $z_2$ only straightening the vector field $g(z_2)\zz$. In this way, the form of $X_1,X_5,X_2$ remains the same, while $X_4$ simplifies to
$$X_4=f(z_2)\z+\zz+(z_1+h(z_2))\zzz$$
(the coefficient functions $f,h,u,v,w$ possibly change). Next, we perform a variable change
$$z_1^*=z_1+\phi(z_2),\quad z_2^*=z_2, \quad z_3^*=z_3+ \psi(z_2).$$
Then it is not difficult to compute that, choosing $\phi,\psi$ as solutions of the system of ODEs $$f+\phi'=0,\quad h+\psi'-\phi=0,$$
we get finally: $$X_4=\zz+z_1\zzz.$$
Using now $[X_2,X_4]=X_4$, it is not difficult to obtain:
$$X_2=(2z_1+A)\z+(B-z_2)\zz+(z_3+Az_2+C)\zzz$$
for some constants $A,B,C$. A shift in $z_2$ allows us to assume further $B=0$.

It remains finally to use the three nontrivial commutation relations containing $X_3$ in \eqref{g5}. A straightforward calculation (the details of which we leave to the reader) give then, first of all, $$A=C=0,$$ and second:
$$X_3=-z_1^2\z+(z_1z_2-z_3)\zz-z_1z_3\zzz.$$
In this way, the initial algebra \eqref{g5} of holomorphic vector field can be brought to the unique normal form given by the above formulas for $X_1,...x_5$.

It remains to integrate the normal form. If $M$ is the orbit of it at some point, then the tangency with $X_1,X_5$ gives that $M$ is given by an equation
$$y_3=F(y_1,x_2,y_2), \quad z_j=x_j+iy_j.$$
The tangency with the three remaining vector fields give the following system of PDEs for $F$:
$$
   2 y_1 \frac{\partial F}{\partial y_1} - x_2 \frac{\partial F}{\partial x_2} - y_2 \frac{\partial F}{\partial y_2} - F = 0,
$$
\begin{equation}\label{system}
   -2 x_1 y_1 \frac{\partial F}{\partial y_1} + (x_1 x_2 - y_1 y_2 - x_3) \frac{\partial F}{\partial x_2} +
       ( x_1 y_2 + x_2 y_1 - F) \frac{\partial F}{\partial y_2}  + (x_1 F + y_1 x_3) = 0,
\end{equation}
$$
    \frac{\partial F}{\partial x_2} - y_1  = 0.
$$
Using the first and the third equations in \eqref{system}, we can simplify the second equation to
\begin{equation}\label{second}
    - y_1^2 y_2  +
       ( x_2 y_1 - F) \frac{\partial F}{\partial y_2}  = 0.
\end{equation}
The third equation in the system yields
$$
   F (y_1, x_2, y_2) = x_2 y_1 + G(y_1, y_2).
$$
Substituting this into the third equation in \eqref{system}, we get
$$GG_{y_2}+y_1^2y_2=0,$$
so that \begin{equation}\label{G2}
G^2=-y_1^2y_2^2+H(y_1).
\end{equation}
Finally, substituting the latter into the first equation in \eqref{system}, we obtain:
$$2y_1G_{y_1}-y_2G_{y_2}-G=0.$$
After multiplying by $G$, by using \eqref{G2}, we get $y_1H'=H$ and so
$$H(y_1)=\alpha y_1,\,\,\alpha\in\RR{*}, \quad F=x_2y_1\pm\sqrt{\alpha y_1-y_1^2y_2^2}.$$
In view of that, the orbit $M$ is an open subset of the real-analytic set
\begin{equation}\label{analset}
(y_3-x_2y_1)^2+y_1^2y_2^2=\alpha y_1.
\end{equation}
It is not difficult to compute that the smooth part of \eqref{analset} is Levi-indefinite for $\alpha\neq 0$, and so is $M$.

If, otherwise, $g(p_2)=0$ in \eqref{xxxx}, we either change the base point $p$ and arrive to the previous case $g(p_2)\neq 0$, or have $g\equiv 0$. In the latter case we conclude, as discussed above, that $v(p_2)\neq 0$ in \eqref{xx}. Arguing as above, we normalize the vector field $X_2$ to become:
$$X_2=2z_1\z+\zz+z_3\zzz.$$
Further, we make use of the substitution $z_2^*=e^{z_2}$ and get:
$$X_2=2z_1\z+\zz+z_3\zzz.$$
Using now the commutation relations for $X_4$ and the fact that $g\equiv 0$ in \eqref{xxxx}, it is straightforward to compute that:
$$X_4=Az_2^3\z+(z_1+Cz_2^2)\zzz.$$
It remains to use the commutation relations for $X_3$. The commutators with $X_1$ and $X_3$ respectively give
\begin{equation}\label{2comm}
\frac{\partial}{\partial z_1}X_3=-X_2,\quad \frac{\partial}{\partial z_1}X_3=-X_2=-X_4.
\end{equation}
Considering finally $[X_2,X_3]=2X_3$, taking the first component of the latter identity and taking \eqref{2comm} into account, it is not difficult to obtain $A=0$. The latter means that the algebra $\g$ contain simultaneously the vector fields $X_5$ and the proportional to it vector field $X_4=(z_1+Cz_2^2)X_5$, which immediately implies the holomorphic degeneracy of the orbit $M$.

We summarize by concluding that there are {\em no} strictly pseudoconvex orbits in the case $\g=g_5$.

\bigskip

We have gone through the entire list of algebras in Mubarakazyanov's classification. Putting together the outcomes in all of the cases above finally proves \autoref{theor1}.

\qed

\section{The classification}
Upon completing the proof of \autoref{theor1}, we are finally able to provide the complete classification of locally homogeneous strictly pseudoconvex hypersurfaces in $\CC{3}$. We would need first of all the following proposition helping to distinguish between two tubular hypersurfaces in our list.

\begin{proposition}\Label{disting}
Let $M_1,M_2\subset\CC{3}$ be two tubular hypersurfaces over affinely homogeneous bases $B_1,B_2$, respectively. Assume further that $M_1,M_2$ are simply homogeneous and that  the  abelian ideal $I$  spanned by the real shifts $i\frac{\partial}{\partial z_j},\,j=1,2,3,$ is the unique $3$-dimensional abelian ideal in both $\g_1$ and $\g_2$. Then $M_1,M_2$ are biholomorphic at some points if and only if their bases are affinely equivalent.
\end{proposition}
\begin{proof}
Assume first there exists  a biholomorphism $H=(f_1,f_2,f_3):\,(M_1,p_1)\lr(M_2,p_2)$. Then, in view of the simple homogeneity, $\g_1$ is mapped onto $\g_2$, and in view of the uniqueness $I$ is mapped into itself. Writing down the fact that the derivations  $\frac{\partial}{\partial z_j},\,j=1,2,3$ are mapped onto (constant) real linear combinations of themselves, we easily conclude that
$\frac{\partial f_k}{\partial z_j}$ are all real constants, so that $H$ is an affine map with a real linear part. Combining with the shifts, we finally get that $H$ is a real affine map. Such a map transforms the bases $B_1,B_2$ onto each other, as follows from the definition of tubular hypersurfaces.

On the other hand, (the complexification of) a real affine map between bases obviously performs an affine equivalence of the tubular manifolds. This proves the proposition.
\end{proof}

\begin{proof}[Proof of \autoref{theor2}]
Applying the results of \cite{lob2dim,lob1dim,dmt}, we can conclude that any locally homogeneous strictly pseudoconvex hypersurface $M\subset\CC{3}$ with $\mbox{dim}\,\mathfrak{aut}\,(M,p)>0$ is locally biholomorphic to one of the hypersurfaces   1) -- 11) considered near a strictly pseudoconvex point in it, and that any two hypersurfaces in the list 1) - 11) are locally biholomorphically inequivalent. In the case $\mbox{dim}\,\mathfrak{aut}\,(M,p)=0$, we apply \autoref{theor1} and conclude that $M$ is locally biholomorphic to the tube over an affinely homogeneous surface in $\RR{3}$. The latter surafces are  classified (locally) by Doubrov-Komrakov-Rabinovich in \cite{dkr} and independently by Ejov-Eastwood in \cite{ee}, up to an affine equivalence. Recall also that, according to \autoref{disting}, the holomorphic classification in the simply homogeneous case is reduced to the affine classification, provided the $3$-dimensional abelian ideal $I\subset\g$ is unique.

Next, note that a tube over a surface in $\RR{3}$ is strictly pseudoconvex iff its base if strictly affinely convex (resp. strictly affinely concave).
Now a straightforward calculation of the second fundamental form for the surfaces in the list in \cite{dkr} allows to exclude from the list in \cite{dkr} all the surfaces violating the strong convexity (resp. strong concavity) condition.

Further, for the resulting list of real hypersurfaces, we exclude those showing up in the lists of hypersurfaces with $\mbox{dim}\,\mathfrak{aut}\,(M,p)>0$ obtained in \cite{lob2dim,lob1dim,dmt}.
This finally gives the list of hypersurfaces 12) -- 17) and proves that any locally homogeneous strcitly pseudoconvex hypersurfaces in $\CC{3}$ is locally equivalent to one of the hypersurfaces 1) -- 17).

As the next step, we need to show that   all the hypersurfaces 12) - 17) indeed have a trivial stability algebra. For doing so, we first note that the family 16) was studied by Beloshapka-Kossovskiy  in \cite{C3} and it was proved there that all the Levi-nondegenerate hypersurfaces in the family have a trivial stabilizer. For hypersurfaces 12) -- 15) and 17), we have to compute the coefficient tensors $N_{22}(0),N_{23}(0)$ in the Chern-Moser normal form  at a strictly pseudoconvex point. As shown in \cite{lob2dim}, a necessary condition for the triviality of the stabilizer is the fact that, in any Chern-Moser normal form, we have $$N_{22}(0)\neq E_0:=|z_1|^4 - 4|z_1|^2 |z_2|^2 + |z_1|^4.$$
A computation employing the MAPLE package shows that, for hypersurfaces 13), 14) and 17), we have $N_{22}(0)\neq E_0$ in any normal form, that is why the latter hypersurfaces have a trivial stabilizer. Next, for hypersurfaces 15) with $\alpha\neq 4$, we similarly have  $N_{22}(0)\neq E_0$ in any normal form, so that the respective stabilizer is trivial. However, for $\alpha=4$, we have $N_{22}(0)= E_0$ in the special normal form, and one has to analyze the tensor $N_{23}(0)$. Not going into further technical details, we again employ the MAPLE package and the results in \cite{lob2dim} and conclude that the tensor $N_{23}(0)$ in the case under discussion contains components contradicting the nontriviality of the stabilizer. Similar situation occurs  for hypersurfaces 12) with $\alpha=\beta=-1$. Namely, we have   $N_{22}(0)= E_0$ in some normal form, while further computations employing the MAPLE package show that the tensor $N_{23}(0)$ contains components contradicting the nontriviality of the stabilizer. In contrast, for $(\alpha,\beta)\neq (-1,-1)$, we have $N_{22}(0)\neq E_0$ in any normal form (by employing MAPLE computations). This finally proves that all hypersurfaces 12) -- 17) have a trivial stibilizer.

It remains to prove that  hypersurfaces 12) -- 17) are all pairwise locally holomorphically inequivalent. Indeed,
it follows directly from the explicit description in \cite{dkr} of the $2$-dimensional affine Lie algebras $\mathfrak a$ acting  on the bases of the surfaces 12) -- 17) that, in each case, $I$ is the unique $3$-dimensional abelian ideal in the Lie algebra $\g$ freely acting on a hypersurface (note that $\g$  equals, as a linear space, to $I\oplus\mathfrak a$). Hence \autoref{disting} is applicable, the equivalence problem is reduced to the affine equivalence problem, and it remains to finally show that the bases of the tubular hypersurfaces 12) -- 17) are pairwise affinely inequivalent. But the latter is accomplished by an elementary 
computation the details of which we leave to the reader. (Alternatively, one can appeal again to the simple homogeneity of the hypersurfaces 12) -- 17) and check that distinct hypersurfaces correspond to distinct algebras in Mubarakazyanov's list; the simple homogeneity now means that any two hypersurfaces as above are locally biholomorphically inequivalent since their automorphism algebras are non-isomorphic). 

The theorem is completely proved now.
\end{proof}

\newpage

\section{Appendix A: Mubarakazyanov's classification of $5$-dimensional real Lie algebras}

\mbox{}

\begin{center} \bf Decomposable $5$-dimensional real Lie algebras \end{center}

\medskip

\begin{equation*}
\begin{array}{|c|c|c|c|c|c|c|c|c|c|c|c|}
\hline
\rotatebox[origin=c]{90}{\mbox{ Алгебры }} & [e_1,e_2] & [e_1,e_3] & [e_1,e_4] & [e_1,e_5] & [e_2,e_3] & [e_2,e_4] & [e_2,e_5] & [e_3,e_4] & [e_3,e_5] & [e_4,e_5] \\ \hline
\mathfrak{m}_{1} & \, & \, & \, & \, & \, & \, & \, & \, & \, & \, \\ \hline
\mathfrak{m}_{2} & e_1 & \, & \, & \, & \, & \, & \, & \, & \, & \, \\ \hline
\mathfrak{m}_{3} & e_1 & \, & \, & \, & \, & \, & \, & e_3 & \, & \, \\ \hline
\mathfrak{m}_{4} & \, & \, & \, & \, & e_1 & \, & \, & \, & \, & \, \\ \hline
\mathfrak{m}_{5} & \, & e_1 & \, & \, & e_1 + e_2 & \, & \, & \, & \, & \, \\ \hline
\mathfrak{m}_{6} & \, & e_1 & \, & \, & e_2 & \, & \, & \, & \, & \, \\ \hline
\mathfrak{m}_{7} & \, & e_1 & \, & \, & he_2 & \, & \, & \, & \, & \, \\ \hline
\mathfrak{m}_{8} & \, & pe_1 - e_2 & \, & \, & e_1 + pe_2 & \, & \, & \, & \, & \, \\ \hline
\mathfrak{m}_{9} & e_1 & 2e_2 & \, & \, & e_3 & \, & \, & \, & \, & \, \\ \hline
\mathfrak{m}_{10} & e_3 & - e_2 & \, & \, & e_1 & \, & \, & \, & \, & \, \\ \hline
\mathfrak{m}_{11} & \, & \, & \, & \, & e_1 & \, & \, & \, & \, & e_4 \\ \hline
\mathfrak{m}_{12} & \, & e_1 & \, & \, & e_1 + e_2 & \, & \, & \, & \, & e_4 \\ \hline
\mathfrak{m}_{13} & \, & e_1 & \, & \, & e_2 & \, & \, & \, & \, & e_4 \\ \hline
\mathfrak{m}_{14} & \, & e_1 & \, & \, & he_2 & \, & \, & \, & \, & e_4 \\ \hline
\mathfrak{m}_{15} & \, & pe_1 - e_2 & \, & \, & e_1 + pe_2 & \, & \, & \, & \, & e_4 \\ \hline
\mathfrak{m}_{16} & e_1 & 2e_2 & \, & \, & e_3 & \, & \, & \, & \, & e_4 \\ \hline
\mathfrak{m}_{17} & e_3 & - e_2 & \, & \, & e_1 & \, & \, & \, & \, & e_4 \\ \hline
\mathfrak{m}_{18} & \, & \, & \, & \, & \, & e_1 & \, & e_2 & \, & \, \\ \hline
\mathfrak{m}_{19} & \, & \, & \alpha e_1 & \, & \, & e_2 & \, & e_2 + e_3 & \, & \, \\ \hline
\mathfrak{m}_{20} & \, & \, & e_1 & \, & \, & \, & \, & e_2 & \, & \, \\ \hline
\mathfrak{m}_{21} & \, & \, & e_1 & \, & \, & e_1 + e_2 & \, & e_2 + e_3 & \, & \, \\ \hline
\mathfrak{m}_{22} & \, & \, & e_1 & \, & \, & \beta e_2 & \, & \gamma e_3 & \, & \, \\ \hline
\mathfrak{m}_{23} & \, & \, & \alpha e_1 & \, & \, & pe_2 - e_3 & \, & e_2 + pe_3 & \, & \, \\ \hline
\mathfrak{m}_{24} & \, & \, & 2e_1 & \, & e_1 & e_2 & \, & e_2 + e_3 & \, & \, \\ \hline
\mathfrak{m}_{25} & \, & \, & (1 + q)e_1 & \, & e_1 & e_2 & \, & qe_3 & \, & \, \\ \hline
\mathfrak{m}_{26} & \, & \, & 2pe_1 & \, & e_1 & pe_2 - e_3 & \, & e_2 + pe_3 & \, & \, \\ \hline
\mathfrak{m}_{27} & \, & e_1 & - e_2 & \, & e_2 & e_1 & \, & \, & \, & \, \\ \hline
\end{array}
\end{equation*}

\medskip

The algebras $\mathfrak{m}_{9},\mathfrak{m}_{10},\mathfrak{m}_{16}$ and $\mathfrak{m}_{17}$ are non-solvable, the others are solvable.

\bigskip

\begin{center} \bf Non-decomposable non-solvabale $5$-dimensional real Lie algebra \end{center}

\medskip

\begin{equation*}
\begin{array}{|c|c|c|c|c|c|c|c|c|c|c|c|}
\hline
\rotatebox[origin=c]{90}{\mbox{ Алгебры }} & [e_1,e_2] & [e_1,e_3] & [e_1,e_4] & [e_1,e_5] & [e_2,e_3] & [e_2,e_4] & [e_2,e_5] & [e_3,e_4] & [e_3,e_5] & [e_4,e_5] \\ \hline
\mathfrak{g}_{5} & 2e_1 & -e_2 &  e_5 & \, & 2 e_3 & e_4 & -e_5 & \, & e_4 & \, \\ \hline
\end{array}
\end{equation*}

\newpage

\begin{center} \bf Non-decomposable solvable $5$-dimensional real Lie algebras \end{center}

\medskip

\begin{equation*}
\begin{array}{|c|c|c|c|c|c|c|c|c|c|c|c|}
\hline
\rotatebox[origin=c]{90}{\mbox{ Алгебры }} & [e_1,e_2] & [e_1,e_3] & [e_1,e_4] & [e_1,e_5] & [e_2,e_3] & [e_2,e_4] & [e_2,e_5] & [e_3,e_4] & [e_3,e_5] & [e_4,e_5] \\ \hline

\mathfrak{g}_{5,1} & \, & \, & \, & \, & \, & \, & \, & \, &  e_1 &  e_2 \\ \hline
\mathfrak{g}_{5,2} & \, & \, & \, & \, & \, & \,  & e_1 & \, &  e_2 &  e_3  \\ \hline
\mathfrak{g}_{5,3} & \, & \, & \, & \, & \, & e_3  & e_1  &  \, & \,  &  e_2  \\ \hline
\mathfrak{g}_{5,4} & \, & \, & \, & \, & \, & e_1 & \, & \, & e_1 &  \,  \\ \hline
\mathfrak{g}_{5,5} & \, & \, & \, & \, & \, &\, & e_1 & e_1 & e_2 & \, \\ \hline
\mathfrak{g}_{5,6} & \, & \, & \, & \, & \, &\, & e_1 & e_1 & e_2 &  e_3 \\ \hline

\mathfrak{g}_{5,7} & \, & \, & \, & e_1 & \, & \, & \alpha e_2 & \, & \beta e_3 & \gamma e_4 \\ \hline
\mathfrak{g}_{5,8} & \, & \, & \, & \, & \, & \, & e_1 & \, & e_3  &\gamma e_4 \\ \hline
\mathfrak{g}_{5,9} & \, & \, & \, & e_1 & \, & \, & e_1 + e_2 & \,  & \beta e_3 & \gamma e_4 \\ \hline
\mathfrak{g}_{5,10} & \, & \, & \, & \, & \, & \, & e_1 & \, & e_2 &  e_4 \\ \hline
\mathfrak{g}_{5,11} & \, & \, & \, & e_1 & \, & \, & e_1 + e_2 & \, & e_2 + e_3 & \gamma e_4 \\ \hline
\mathfrak{g}_{5,12} & \, & \, & \, & e_1 & \, & \, & e_1 + e_2 & \, & e_2 + e_3 & e_3 + e_4 \\ \hline
\mathfrak{g}_{5,13} & \, & \, & \, & e_1 & \, & \, & \gamma e_2 & \, & p e_3- s e_4  & s e_3 + p e_4 \\ \hline
\mathfrak{g}_{5,14} & \, & \, & \, & \, & \, & \, & e_1 & \, & p e_3 - e_4 & e_3 + p e_4 \\ \hline
\mathfrak{g}_{5,15} & \, & \, & \, & e_1 & \, & \, & e_1 + e_2 & \, & \gamma e_3 & e_3 + \gamma e_4 \\ \hline
\mathfrak{g}_{5,16} & \, & \, & \, & e_1 & \, & \, & e_1 + e_2 & \, & p e_3- s e_4  & s e_3 + p e_4 \\ \hline
\mathfrak{g}_{5,17} & \, & \, & \, & p e_1 - e_2 & \, & \, &e_1 + p e_2 & \, & q e_3- s e_4  & s e_3 + q e_4 \\ \hline
\mathfrak{g}_{5,18} & \, & \, & \, & p e_1 - e_2 & \, & \, & e_1 + p e_2 & \, & e_1 + p e_3- e_4  & e_2 + e_3 - p e_4 \\ \hline
\mathfrak{g}_{5,19} & \, & \, & \, & (1+\alpha)e_1 & e_1 & \, &  e_2 & \, & \alpha e_3 & \beta e_4 \\ \hline
\mathfrak{g}_{5,20} & \, & \, & \, & (1+\alpha)e_1 & e_1 & \, & e_2 & \, & \alpha e_3  & e_1 +(1+ \alpha) e_4 \\ \hline
\mathfrak{g}_{5,21} & \, & \, & \, & 2 e_1 & e_1 & \, & e_2 + e_3 & \, & e_3 + e_4 & e_4 \\ \hline
\mathfrak{g}_{5,22} & \, & \, & \, & \, & e_1 & \, & e_3 & \, & \, &  e_4 \\ \hline
\mathfrak{g}_{5,23} & \, & \, & \, & 2e_1 & e_1 & \, & e_2 + e_3 & \, & e_3 & \beta e_4 \\ \hline
\mathfrak{g}_{5,24} & \, & \, & \, & 2e_1 & e_1 & \, & e_2 + e_3 & \, & e_3 & \varepsilon e_1 + 2 e_4 \\ \hline
\mathfrak{g}_{5,25} & \, & \, & \, & 2pe_1 & e_1 & \, & p e_2 + e_3 & \, & -e_2 +p e_3 & \beta e_4 \\ \hline
\mathfrak{g}_{5,26} & \, & \, & \, & 2pe_1 & e_1 & \, & p e_2 + e_3 & \, & -e_2 +p e_3 & \varepsilon e_1 + 2 p e_4 \\ \hline
\mathfrak{g}_{5,27} & \, & \, & \, & e_1 & e_1 & \, & \, & \, & e_3 + e_4 & e_1 + e_4 \\ \hline
\mathfrak{g}_{5,28} & \, & \, & \, & (1+\alpha)e_1 & e_1 & \, & \alpha e_2 & \, & e_3 + e_4 &  e_4 \\ \hline
\mathfrak{g}_{5,29} & \, & \, & \, & e_1 & e_1 & \, & e_2 & \, & e_4 & \, \\ \hline
\mathfrak{g}_{5,30} & \, & \, & \, & (2 +h)e_1 & \, &  e_1 & ( 1+h) e_2 & e_2  & h e_3 & e_4 \\ \hline
\mathfrak{g}_{5,31} & \, & \, & \, & 3 e_1 &  \,& e_1  &  2 e_2 &  e_2  & e_3  & e_3 + e_4 \\ \hline
\mathfrak{g}_{5,32} & \, & \, & \, & e_1 & \,  & e_1  & e_2 & e_2 & h e_1 + e_3 &  \, \\ \hline
\mathfrak{g}_{5,33} & \, & \, & e_1  & \, & \, & \, & e_2 & \beta e_3 & \gamma e_3 & \, \\ \hline
\mathfrak{g}_{5,34} & \, & \, & \alpha e_1  & e_1& \, & e_2 & \, & e_3 & e_2  & \, \\ \hline
\mathfrak{g}_{5,35} & \, & \, & h e_1  & \alpha e_1 & \, & e_2 & - e_3 & e_3  &  e_2 & \, \\ \hline
\mathfrak{g}_{5,36} & \, & \, & e_1  & \, & e_1 & e_2 & -e_2 & \, & e_3 &  \, \\ \hline
\mathfrak{g}_{5,37} & \, & \, & 2 e_1  & \, & e_1 &e_2 & -e_3 & e_3 & e_2 & \, \\ \hline
\mathfrak{g}_{5,38} & \, & \, & e_1  & \, & \, & \, & e_2 & \, & \, & e_3 \\ \hline
\mathfrak{g}_{5,39} & \, & \, & e_1  &  -e_2 & \, & e_2 & \,e_1 &\, & \,  & e_3 \\ \hline
\end{array}
\end{equation*}

\end{document}